\documentclass[a4paper,12pt]{article}
\usepackage{cmap}                        
\usepackage[cp1251]{inputenc}            
\usepackage[english]{babel}
\usepackage[left=2cm,right=2cm,top=2cm,bottom=2cm]{geometry} 
\usepackage{amssymb}
\usepackage{amsmath, amsthm}
\theoremstyle{plain}
\newtheorem{theorem}{Theorem}
\newtheorem{lemma}{Lemma}
\newtheorem{proposition}{Proposition}
\newtheorem{corollary}{Corollary}[theorem]

\theoremstyle{definition}
\newtheorem{example}{Example}
\newtheorem{remark}{Remark}
\newtheorem{definition}{Definition}
\newtheorem{pr}{Question}

\begin{document}

\begin{center}\Large
\textbf{On one question of  Shemetkov  about composition formations }
\normalsize

\smallskip
V.\,I. Murashka

 \{mvimath@yandex.ru\}

 Francisk Skorina Gomel State University, Gomel, Belarus\end{center}

\textbf{Abstract.} In this paper one construction of composition formations was introduced. This construction contains formations of quasinilpotent  groups, $c$-supersoluble groups, groups defined by ranks of chief factors and some new classes of groups. A partial answer on a question of L.\,A. Shemetkov about the intersection of $\mathfrak{F}$-maximal subgroups and the $\mathfrak{F}$-hypercenter was given for these composition formations.


\textbf{AMS}(2010). 20D25,  20F17,   20F19.

\section*{Introduction}

\emph{All groups considered here will be finite.}
 Through   $G$, $p$ and $\mathfrak{X}$   we always denote here respectively a finite group, a prime and a class of groups.

One of the directions in the modern group theory is to \textbf{construct classes of groups} and \textbf{study the properties of all groups in such class} (formation, Fitting class, Schunk class and etc.)

Recall that a \emph{formation} is a class $\mathfrak{X}$ of groups with the following properties:
$(a)$ every homomorphic image of an $\mathfrak{X}$-group is an $\mathfrak{X}$-group, and
$(b)$ if $G / M$ and $G / N$ are $\mathfrak{X}$-groups, then also $G/(M\cap N)\in \mathfrak{X}$.

One of the main classes of formations is the class of local formations. Recall that a function of the form $f: \mathbb{P}\rightarrow\{formations\}$ is called a \emph{formation function} and   a formation $\mathfrak{F}$ is called \emph{local} \cite[IV, 3.1]{s8} if
\begin{center} $\mathfrak{F}=(G\,|\, G/C_G(\overline{H})\in f(p)$ for every $p\in\pi(\overline{H})$ and every chief factor $\overline{H}$ of $G$)
\end{center}
for some formation function $f$. In this case $f$ is called a \emph{local definition} of $\mathfrak{F}$.
By the Gasch\"utz-Lubeseder-Schmid theorem, a  formation is local if and only if it is non-empty and \emph{saturated}, i.e. from $G/\Phi(G)\in\mathfrak{F}$ it follows that $G\in\mathfrak{F}$ where $\Phi(G)$ is the Frattini subgroup of $G$.    The classes of all unit $ \mathfrak{E}$, nilpotent $\mathfrak{N}$, metanilpotent $\mathfrak{N}^2$, supersoluble $\mathfrak{U}$ and soluble groups $\mathfrak{S}$ are examples of local formations.

The most applications of local formations are  in the theory of soluble groups. Let us mention another interesting class of formations of soluble groups. Let $\overline{N}$ be a chief factor of $G$. Then $\overline{N}=\overline{N}_1\times\dots\times \overline{N}_n$ where $\overline{N}_i$ are isomorphic simple groups. The number $n=r(\overline{N}, G)$ is the \emph{rank} of $\overline{N}$ in $G$. A \emph{rank function} $R$ \cite[VII, 2.3]{s8} is a map which associates with each prime $p$
a set $R(p)$ of natural numbers. For each rank function let
\begin{center}
  $\mathfrak{E}(R)=(G\in\mathfrak{S}\,|$ for all $p\in \mathbb{P}$ each $p$-chief factor of $G$ has rank in $R(p)$).
\end{center}
Note that $\mathfrak{E}(R)$ is a formation.
Heineken \cite{r1} and Harman \cite{r2} characterized all rank functions $R$ for which  $\mathfrak{E}(R)$ is local. Analogues questions for formations not of full characteristic were studied by Huppert \cite{r3}, Kohler \cite{r4} and Harman \cite{r1}.  Haberl and Heineken \cite{r5} characterized all rank functions $R$ for which  $\mathfrak{E}(R)$ is a Fitting formation.


 A function of the form $f: \mathbb{P}\cup\{0\}\rightarrow\{formations\}$ is called a composition definition. Recall \cite[p. 4]{s5} that  a formation $\mathfrak{F}$ is called \emph{composition} or \emph{Baer-local} if
\begin{center}
$\mathfrak{F}=(G\,|\, G/G_\mathfrak{S}\in f(0)$ and $G/C_G(\overline{H})\in f(p)$  for every  abelian $p$-chief factor $\overline{H}$ of $G$)
\end{center}
for some composition definition $f$. A formation is composition (Baer-local) \cite[IV, 4.17]{s8}  if and only if it is \emph{solubly saturated}, i.e. from $G/\Phi(G_\mathfrak{S})\in\mathfrak{F}$ it follows that $G\in\mathfrak{F}$, where $G_\mathfrak{S}$ is the soluble radical of $G$.

Note that a local formation is a composition formation. The converse is false. An example of nonlocal composition formation is the class of all quasinilpotent groups $\mathfrak{N}^*$ that was introduced by Bender \cite{18}.

Recall that a chief factor $\overline{H}$ of  $G$ is called   $\mathfrak{X}$-\emph{central} in $G$ provided    $\overline{H}\rtimes G/C_G(\overline{H})\in\mathfrak{X}$ (see \cite[p. 127--128]{s6}), otherwise it is called $\mathfrak{X}$-\emph{eccentric}. The symbol $\mathrm{Z}_\mathfrak{X}(G)$ denotes the $\mathfrak{X}$-\emph{hypercenter} of $G$, that is, the largest normal subgroup of $G$
such that every chief factor $\overline{H}$ of $G$ below it is $\mathfrak{X}$-central.  If $\mathfrak{X}=\mathfrak{N}$ is the class of all nilpotent groups, then $\mathrm{Z}_\mathfrak{N}(G)=\mathrm{Z}_\infty(G)$ is the hypercenter of $G$. If $\mathfrak{F}$ is a composition formation, then by {\cite[1, 2.6]{s5}}
$$\mathfrak{F}=(G\,|\,Z_\mathfrak{F}(G)=G). $$

The general definition of a composition formation $\mathfrak{F}$  gives little information about the action of an $\mathfrak{F}$-group $G$ on its non-abelian chief factors. Therefore several families of composition formations were introduced by giving additional information about the action of an $\mathfrak{F}$-group on its non-abelian chief factors.  For example, in \cite{sk2, sk1}  Guo and Skiba introduced the class $\mathfrak{F}^*$ of  quasi-$\mathfrak{F}$-groups for a saturated formation $\mathfrak{F}$:
 \begin{center}
   $\mathfrak{F}^*=(G\,|$ for every $\mathfrak{F}$-eccentric chief factor $\overline{H}$ and every $x\in G$, $x$ induces an inner automorphism on $\overline{H}$).
 \end{center}
  If $\mathfrak{N}\subseteq\mathfrak{F}$ is a normally hereditary saturated formation, then $\mathfrak{F}^*$  is a normally hereditary solubly saturated formation by \cite[Theorem 2.6]{sk2}.

Another example of a nonlocal composition formation is the class of all $c$-supersoluble groups that was introduced by Vedernikov in \cite{j0}.  Recall that a group is called $c$-\emph{supersoluble} ($SC$-group in the terminology  of   Robinson \cite{j1}) if every its chief factor is a simple   group.   The products of $c$-supersoluble groups were studied in  \cite{j3, j9, j10, j11, j4}.
 According to \cite{j8}, a group  $G$ is called  $\mathfrak{J}c$-\emph{supersoluble} if every chief $\mathfrak{J}$-factor of $G$ is a simple group, where $\mathfrak{J}$ is a class of simple groups. In \cite{j8} the same idea was applied for some other classes of groups.  The products and properties of such groups were studied in \cite{j8, j6}. 

In   \cite{j5} the class $\mathfrak{F}_{ca}$ of $ca$-$\mathfrak{F}$-groups was introduced:
\begin{center}
  $\mathfrak{F}_{ca}=(G\,|$  abelian chief factors of $G$ are $\mathfrak{F}$-central and other chief factors are simple groups).
\end{center}
 The class $\mathfrak{F}_{ca}$ was studied in \cite{j5, j7}, where $\mathfrak{F}$ is a saturated formation. In particular it is a composition formation.

In this paper we generalize constructions of quasi-$\mathfrak{F}$-groups, $ ca$-$\mathfrak{F}$-groups, $\mathfrak{J}c$-supersoluble groups and groups  defined by the rank function in the sense of the following definition.

\begin{definition}\label{MD1}
$(1)$ A generalized rank function $\mathcal{R}$ is a map defined on direct products of isomorphic simple groups by

$(a)$ $\mathcal{R}$ associates with  each simple group $S$ a pair $\mathcal{R}(S)=(A_\mathcal{R}(S), B_\mathcal{R}(S))$ of possibly empty disjoint sets $A_\mathcal{R}(S)$ and $B_\mathcal{R}(S)$ of natural numbers.

$(b)$ If $N$ is  the direct products of  simple isomorphic to $S$ groups, then $\mathcal{R}(N)=\mathcal{R}(S)$.

$(2)$ Let $\overline{N}$ be a chief factor of $G$.   We shall say that  a \emph{generalized rank} of $\overline{N}$ in $G$  lies in $\mathcal{R}(\overline{N})$ (briefly $gr(\overline{N}, G)\in \mathcal{R}(\overline{N}))$ if  $r(\overline{N}, G)\in A_\mathcal{R}(\overline{N})$ or $r(\overline{N}, G)\in B_\mathcal{R}(\overline{N})$  and if some  $x\in G$ fixes a composition factor $\overline{H}/\overline{K}$ of $\overline{N}$ (i.e. $\overline{H}^x=\overline{H}$ and $\overline{K}^x=\overline{K}$), then  $x$ induces an inner automorphism on it.

$(3)$ With each generalized rank function $\mathcal{R}$ and a class of groups $\mathfrak{X}$ we associate a class
  $$\mathfrak{X}(\mathcal{R})=(G\,|\,\overline{H}\not\in\mathfrak{X}\textrm{ and }gr(\overline{H}, G)\in \mathcal{R}(\overline{H})\textrm{ for every } \mathfrak{X}\textrm{-eccentric chief factor } \overline{H}\textrm{ of }G)$$
\end{definition}


\begin{example}\label{ex1}
A lot of above mentioned formations can be described with the help of our construction:

\begin{enumerate}
\item Let $\mathfrak{E}=(1)$. Assume that   $\mathcal{R}(H)=(\{1\},\emptyset)$ if  $H$ is abelian and $\mathcal{R}(H)=(\emptyset, \emptyset)$ otherwise. Then $\mathfrak{E}(\mathcal{R})=\mathfrak{U}$.

\item If $\mathcal{R}(H)\equiv(\{1\}, \emptyset)$, then $\mathfrak{E}(\mathcal{R})$ is the class $\mathfrak{U}_c$ of all $c$-supersoluble groups.

\item Let $\mathfrak{J}$ be a class of simple groups.  If $\mathcal{R}(H)\equiv(\{1\}, \emptyset)$ for $H\in\mathfrak{J}$ and $\mathcal{R}(H)=(\mathbb{N}, \emptyset)$ otherwise, then $\mathfrak{E}(\mathcal{R})$ is the class of all $\mathfrak{J}c$-supersoluble groups.

\item Assume that $\mathcal{R}(H)=(A_\mathcal{R}(H), \emptyset)$ if  $H$ is abelian and $\mathcal{R}(H)=(\emptyset, \emptyset)$ otherwise. Then $\mathcal{R}$ is a rank function.

 \item Let    $\mathcal{R}(H)=(\emptyset, \{1\})$ if  $H$ is abelian and $\mathcal{R}(H)=(\emptyset, \emptyset)$ otherwise. Then $\mathfrak{E}(\mathcal{R})=\mathfrak{N}$.

 \item If   $\mathcal{R}(H)\equiv(\emptyset, \{1\})$, then $\mathfrak{E}(\mathcal{R})=\mathfrak{N}^*$.

 \item Assume that   $\mathcal{R}(H)=(\emptyset, \{1\})$ if  $H$ is abelian and $\mathcal{R}(H)=(\{1\}, \emptyset)$ otherwise. Then $\mathfrak{E}(\mathcal{R})=\mathfrak{N}_{ca}$.

\item Let $\mathfrak{N}\subseteq\mathfrak{F}$ be a normally hereditary saturated formation. If   $\mathcal{R}(H)\equiv(\emptyset, \{1\})$, then $\mathfrak{F}(\mathcal{R})=\mathfrak{F}^*$ (see the proof of Corollary \ref{c13}).

\item Let $\mathfrak{F}\subseteq\mathfrak{S}$ be a normally hereditary saturated formation, $\mathcal{R}(H)=(\emptyset, \emptyset)$ for abelian $H\not\in\mathfrak{F}$ and $\mathcal{R}(H)=(\{1\},\emptyset)$ for non-abelian $H$. Then $\mathfrak{F}(\mathcal{R})=\mathfrak{F}_{ca}$.
\end{enumerate}
\end{example}
Recall that a subgroup $U$ of  $G$ is called $\mathfrak{X}$-\emph{maximal} in $G$
provided that $(a)$ $U\in\mathfrak{X}$, and $(b)$ if $U\leq V \leq G$ and $V\in\mathfrak{X}$, then $U = V$ \cite[p. 288]{s8}. The symbol $\mathrm{Int}_\mathfrak{X}(G)$
denotes the intersection of all $\mathfrak{X}$-maximal subgroups of $G$.

Note that the intersection of maximal abelian subgroups of $G$ is the center of $G$. According to     Baer \cite{h1},   the intersection of maximal nilpotent subgroups of $G$ coincides with the hypercenter of $G$.   In \cite[Example 5.17]{h4} it was shown that  the intersection of maximal supersoluble subgroups of $G$ does not necessary coincide with the  supersoluble hypercenter of $G$.
 Shemetkov possed the following question on  Gomel Algebraic seminar in 1995:

\begin{pr}\label{q1}    For what non-empty normally hereditary
solubly saturated formations $\mathfrak{X}$ do the equality
$\mathrm{Int}_\mathfrak{X}(G)=\mathrm{Z}_\mathfrak{X}(G)$ hold for every group $G$?\end{pr}

The solution to this  question for hereditary saturated formations was obtained by  Skiba in   \cite{h4, h5} (for the soluble case, see also  Beidleman  and Heineken \cite{h3}) and for the class of all quasi-$\mathfrak{F}$-groups, where $\mathfrak{F}$ is a hereditary saturated formation, was given  in \cite{ArX}. In particular, the intersection of maximal quasinilpotent subgroups is the quasinilpotent hypercenter. The aim of this paper is to give the answer on this  question for $\mathfrak{F}(\mathcal{R})$.

\section*{Preliminaries}

The notation and terminology agree with the books \cite{s8, s5}. We refer the reader to these
books for the results on formations. Recall that $G^\mathfrak{F}$ is the $\mathfrak{F}$-residual of $G$ for a formation $\mathfrak{F}$;  $G_\mathfrak{S}$ is the soluble radical of  $G$; $\tilde{\mathrm{F}}(G)$ is defined by $\tilde{\mathrm{F}}(G)/\Phi(G)=\mathrm{Soc}(G/\Phi(G))$; $\pi(G)$ is the set of all prime divisors of  $G$;  $\pi(\mathfrak{X})=\underset{G\in\mathfrak{X}}\cup\pi(G)$; $\mathfrak{N}_p\mathfrak{F}=(G\,|\, G/\mathrm{O}_p(G)\in \mathfrak{F})$ is a formation for a formation $\mathfrak{F}$; $G$ is called $s$-critical for $\mathfrak{X}$ if all proper subgroups of $G$ are $\mathfrak{X}$-groups and  $G\not\in\mathfrak{X}$; $\mathrm{Aut}G$, $\mathrm{Inn}G$ and $\mathrm{Out}G$ are respectively groups of all, inner and outer automorphisms of $G$; $E\mathfrak{X}$ is the class of groups all whose composition factors are $\mathfrak{X}$-groups; $N\wr S_n$ is the natural wreath product of $N$ and the symmetric group $S_n$ of degree $n$.

\section{Main Results}

\subsection{The canonical composition definition of $\mathfrak{F}(\mathcal{R})$}

Recall that any nonempty composition formation $\mathfrak{F}$ has an unique composition definition $F$
such that $F(p)=\mathfrak{N}_pF(p)\subseteq\mathfrak{F}$ for all primes $p$ and  $F(0) = \mathfrak{F}$ (see \cite[1, 1.6]{s5}). In this case $F$ is called the \emph{canonical composition definition}  of $\mathfrak{F}$.  We shall say that a generalized rank function $\mathcal{R}$ is (resp. \emph{very}) \emph{good} if for any simple group $S$ holds:

 $(a)$ from $a\in A_\mathcal{R}(S)$ it follows that $b\in A_\mathcal{R}(S)$ for any natural $b|a$ (resp. $b\leq a$);

$(b)$ from $a\in B_\mathcal{R}(S)$ it follows that $b\in A_\mathcal{R}(S)\cup B_\mathcal{R}(S)$  (resp.$b\in B_\mathcal{R}(S)$) for any natural $b|a$ (resp. $b\leq a)$.

\begin{theorem}\label{MT1}
  Let $\mathfrak{N}\subseteq\mathfrak{F}$ be a composition formation with the canonical composition definition $F$ and $\mathcal{R}$ be a generalized rank function.  Then

  $(1)$  $\mathfrak{F}(\mathcal{R})$  is a composition formation with the canonical composition definition $F_\mathcal{R}$ such that $F_{\mathcal{R}}(0)=\mathfrak{F}(\mathcal{R})$ and $F_{\mathcal{R}}(p)=F(p)$ for all $p\in\mathbb{P}$.

  $(2)$ If $\mathfrak{F}$ is normally hereditary and $\mathcal{R}$ is good, then $\mathfrak{F}(\mathcal{R})$ is normally hereditary.

\end{theorem}

\begin{corollary}[{\cite[Theorem 1]{j3}}]\label{c11}
  $\mathfrak{U}_c$ is a composition formation with the canonical composition definition  $h$ such that $h(p)=\mathfrak{N}_p\mathfrak{A}(p-1)$ for every prime $p$ and $h(0)=\mathfrak{U}_c$.
\end{corollary}

In  \cite{V1} the class $w\mathfrak{U}$ of widely supersoluble groups was introduced. It is a hereditary saturated formation of soluble groups. Recall \cite{j4} that a group  is called widely  $c$-supersoluble if its abelian chief factors are $w\mathfrak{U}$-central and other chief factors are simple groups.

\begin{corollary}[{\cite[Theorem A]{j4}}]\label{c12} The class
  $\mathfrak{U}_{cw}$ of widely  $c$-supersoluble groups is a normally hereditary composition formation with the canonical composition definition  $h$ such that $h(p)=\mathfrak{N}_p(G | G\in w\mathfrak{U}\cap \mathfrak{N}_p\mathcal{A}(p-1))$ for every prime $p$ and $h(0)=\mathfrak{U}_{cw}$.
\end{corollary}

\begin{corollary}[{\cite[Theorem 2.6]{sk2}}]\label{c13}
 For every saturated formation $\mathfrak{F}$ containing all nilpotent groups with the canonical local definition $F$, the formation $\mathfrak{F}^*$ is composition with the canonical composition definition $F^*$ where  $F^*(p)=F(p)$ for every prime $p$ and $F^*(0)=\mathfrak{F}^*$. Moreover, if the formation $\mathfrak{F}$  is normally hereditary, then $\mathfrak{F}^*$ is also normally hereditary.
\end{corollary}


In the proof of Theorem \ref{MT1} we will need the following lemmas:

\begin{lemma}\label{Lgr}
  Let $H/K$ and $M/N$ be $G$-isomorphic chief factors of $G$.

  $(a)$ Then they have the same generalized rank.

  $(b)$ \cite[1, 1.14]{s5} $H/K\rtimes G/C_G(H/K)\simeq M/N\rtimes G/C_G(M/N)$.
\end{lemma}

\begin{proof}
  Let $\alpha: H/K\to M/N$ be a $G$-isomorphism.  Since $H/K$ and $M/N$  are isomorphic groups, they have the same rank.  Assume that    $x\in G$ fixes a composition factor $A/B$ of $H/K$ and   induces an inner automorphism $aB$ on it. Note that $\alpha(A/B)$ is a composition factor of $M/N$. From $\alpha(A/B)^x=\alpha(A/B^x)=\alpha(A/B)$ it follows that $x$ fixes $\alpha(A/B)$ and it is straightforward to check that $x$ induces an inner  automorphism $\alpha(aB)$ on it. Since $\alpha^{-1}$ is also $G$-isomorphism, we see that the generalized ranks of  $H/K$ and $M/N$  are the same.
\end{proof}

\begin{lemma}[{\cite[1, 1.15]{s5}}]\label{l1}
Let $\overline{H}$ be a chief factor of  $G$. Then

$(1)$ If $\mathfrak{F}$ is a composition formation and $F$ is its canonical composition definition, then $\overline{H}$   is $\mathfrak{F}$-central if and only if $G/C_G(\overline{H})\in F(p)$ for all $p\in\pi(\overline{H})$  in the case when $\overline{H}$ is abelian, and $G/C_G(\overline{H})\in\mathfrak{F}$ when $\overline{H}$ is non-anelian.

$(2)$ If $\mathfrak{F}$ is a local formation and $F$ is its canonical local definition, then $\overline{H}$   is $\mathfrak{F}$-central if and only if $G/C_G(\overline{H})\in F(p)$ for all $p\in\pi(\overline{H})$.
\end{lemma}

\begin{lemma}[{\cite[1, 2.6]{s5}}]\label{lz} Let $\mathfrak{F}$ be a solubly saturated formation. Then
 $\mathfrak{F}=(G\,|\, G=\mathrm{Z}_\mathfrak{F}(G))$.
\end{lemma}

Recall that $C^p(G)$ is the intersection of the centralizers of all abelian $p$-chief factors of   $G$ ($C^p(G)=G$ if $G$ has no such chief factors). Let $f$ be a composition definition of a composition formation $\mathfrak{F}$. It is known that  $\mathfrak{F}=(G\,|\, G/G_\mathfrak{S}\in f(0)$ and $G/C^p(G)\in f(p)$ for all $p\in\pi(G)$ such that $G$ has an abelian $p$-chief factor).

\begin{lemma}[{\cite[X, 13.16(a)]{Hup}}]\label{ls5}
Suppose that $G =G_1\times\dots\times G_n$, where each $G_i$  is a simple
non-abelian normal subgroup of $G$  and $G_i\neq G_j$  for $i\neq j$. Then   any subnormal subgroup $H$ of $G$ is the direct product
of certain $G_i$.
\end{lemma}

The following lemma directly  follows from  previous lemma

\begin{lemma}\label{l3}
Let a normal subgroup   $N$ of  $G$ be a direct product of isomorphic simple non-abelian groups. Then  $N$ is a direct product of minimal normal subgroups of $G$.
\end{lemma}

\begin{proof}[{\bf Proof of Theorem \ref{MT1}}] $(1)$
 From $(b)$ and $(c)$ of the Isomorphism Theorems \cite[2.1A]{s8} and Lemma \ref{Lgr} it follows that $\mathfrak{X}(\mathcal{R})$ is a formation for any class of groups $\mathfrak{X}$.
So, $\mathfrak{F}(\mathcal{R})$ is a formation.
  Let $\mathfrak{H}=CLF(F_{\mathcal{R}})$.

  Assume  $\mathfrak{H}\setminus\mathfrak{F}({\mathcal{R}})\neq\emptyset$. Let chose a minimal order group $G$ from $\mathfrak{H}\setminus\mathfrak{F}({\mathcal{R}})$.
  Since $\mathfrak{F}({\mathcal{R}})$ is a formation, $G$ has an unique minimal normal subgroup $N$ and $G/N\in\mathfrak{F}({\mathcal{R}})$.

  Suppose that $N$ is abelian. Then it is a $p$-group. Since $N$ is $\mathfrak{H}$-central in $G$ by Lemma \ref{lz}, $G/C_G(N)\in F_{\mathcal{R}}(p)$ by Lemma \ref{l1}. From $F_{\mathcal{R}}(p)=F(p)$ and Lemma \ref{l1} it follows that $N$ is an $\mathfrak{F}$-central chief factor of $G$. Hence $G\in\mathfrak{F}(\mathcal{R})$, a contradiction.

  So  $N$ is non-abelian. Note that $G_\mathfrak{S}\leq C_G(N)$ by \cite[1, 1.5]{s5}. Hence $G\simeq G/C_G(N)\in F_{\mathcal{R}}(0)=\mathfrak{F}(\mathcal{R})$, the contradiction. Thus $\mathfrak{H}\subseteq\mathfrak{F}(\mathcal{R})$.

Assume  $\mathfrak{F}(\mathcal{R})\setminus\mathfrak{H}\neq\emptyset$. Let chose a minimal order group $G$ from $\mathfrak{F}(\mathcal{R})\setminus\mathfrak{H}$. Since $\mathfrak{H}$ is a formation, $G$ has an unique minimal normal subgroup $N$ and $G/N\in\mathfrak{H}$.

If $N$ is abelian, then $G/C_G(N)\in F(p)$ for some $p$ by Lemmas \ref{l1} and \ref{lz}. From $F_{\mathcal{R}}(p)=F(p)$ and Lemma \ref{l1} it follows that $N$ is $\mathfrak{H}$-central in $G$. So $G\in\mathfrak{H}$, a contradiction.

 Hence $N$ is non-abelian.  It means that $G_\mathfrak{S}=1$. Therefore $G/G_\mathfrak{S}\simeq G\in\mathfrak{F}(\mathcal{R})=F_\mathcal{R}(0)$. Note that $N\leq C^p(G)$ for all primes $p$. So $C^p(G)/N=C^p(G/N)$. From $G/N\in\mathfrak{H}$ it follows that $G/C^p(G)\simeq (G/N)/C^p(G/N)\in F_\mathcal{R}(p)$ for any $p$ such that $G$ has an abelian chief $p$-factor. Therefore $G\in\mathfrak{H}$, the contradiction. So $\mathfrak{F}(\mathcal{R})\subseteq \mathfrak{H}$. Thus   $\mathfrak{F}(\mathcal{R})=\mathfrak{H}$.

 $(2)$ Let  $F$ be the canonical composition definition of $\mathfrak{F}$, $G$ be an $\mathfrak{F}$-group and     $1=N_0\trianglelefteq N_1\trianglelefteq\dots\trianglelefteq N_n=N\trianglelefteq G$ be the part of chief series of $G$ below $N$. Let $H/K$ be a chief factor of $N$ such that
 $N_{i-1}\leq K\leq H\leq N_i$ for some $i$.

 If $N_i/N_{i-1}\not\in \mathfrak{F}$, then it is a non-abelian group. According to Lemma \ref{l3} $N_i/N_{i-1}$ is a direct product of minimal normal subgroups of $N/N_{i-1}$.  Let $L/N_{i-1}$ be one of them and $L_1/N_{i-1}$ be a direct simple factor of it. Note that $r(N_i/N_{i-1}, G)=|G:N_G(L_1/N_{i-1})|$, $N_G(L_1/N_{i-1})\cap N=N_N(L_1/N_{i-1})$ and $|G:N_N(L_1/N_{i-1})|$ is a divisor of $|G:N_G(L_1/N_{i-1})|$ by \cite[\S1, Lemma 1]{h1}.    It means  that $r(L/N_{i-1}, N)$ divides $  r(N_i/N_{i-1}, G)$ and every composition factor of $L/N_{i-1}$ is a composition factor of $N_i/N_{i-1}$. Since $\mathcal{R}$ is a good generalized rank function, $gr(L/N_{i-1}, N)\in \mathcal{R}(L/N_{i-1})$   for any chief factor $L/N_{i-1}$ of $N$ between  $N_{i-1}$ and $N_i$.

If $N_i/N_{i-1}\in \mathfrak{F}$, then it is $\mathfrak{F}$-central in $G$.   Note that $H/K\in\mathfrak{F}$.

 Assume that $N_i/N_{i-1}$ is abelian.  Then $G/C_G(N_i/N_{i-1})\in F(p)$ for some $p$  by Lemma \ref{l1}. Note that $F(p)$ is a normally hereditary formation by \cite[IV, 3.16]{s8}. Since $$NC_G(N_i/N_{i-1})/C_G(N_i/N_{i-1})\trianglelefteq G/C_G(N_i/N_{i-1}),$$ we see that $$NC_G(N_i/N_{i-1})/C_G(N_i/N_{i-1})\simeq N/C_N(N_i/N_{i-1})\in F(p).$$   From $C_N(N_i/N_{i-1})\leq C_N(H/K)$ it follows that $N/C_N(H/K)$ is a quotient group of \linebreak $N/C_N(N_i/N_{i-1})$. Thus $N/C_N(H/K)\in F(p)$.
Now $H/K$ is an $\mathfrak{F}$-central chief factor of $N$ by Lemma \ref{l1}.

Assume that $N_i/N_{i-1}$ is non-abelian. Then $G/C_G(N_i/N_{i-1})\in\mathfrak{F}$ by Lemma \ref{l1}. Hence $NC_G(N_i/N_{i-1})/C_G(N_i/N_{i-1})\in\mathfrak{F}$. By analogy $N/C_N(H/K)\in \mathfrak{F}$. So $H/K$  is an $\mathfrak{F}$-central chief factor of $N$ by Lemma \ref{l1}.

  Thus every chief $\mathfrak{F}$-factor of $N$ is $\mathfrak{F}$-central and  $gr(\overline{H}, N)\in \mathcal{R}(\overline{H})$  for other chief factors $\overline{H}$ of $N$   by Jordan-H\"{o}lder theorem. Thus $N\in\mathfrak{F}(\mathcal{R})$.
\end{proof}

\begin{proof}[{\bf Proof of Corollaries \ref{c11}, \ref{c12} and \ref{c13}}]
Recall that every local formation is composition. It is known that if $F$ is the canonical local definition of a local formation $\mathfrak{F}$, then $D$ is the canonical composition definition of $\mathfrak{F}$ where $D(0)=\mathfrak{F}$ and $F(p)=D(p)$ for all prime $p$.

Let $\mathcal{R}(H)\equiv(\{1\},\emptyset)$. Note that $\mathcal{R}$ is good. Recall that the classes of all supersoluble and widely supersoluble groups are hereditary local formations with the canonical local definitions $F(p)=\mathfrak{N}_p\mathfrak{A}(p-1)$ and $D(p)=\mathfrak{N}_p(G | G\in w\mathfrak{U}\cap \mathfrak{N}_p\mathcal{A}(p-1))$ (see \cite[Lemma 3.2]{j4}) for every prime $p$ respectively. Note that $\mathfrak{U}_c=\mathfrak{U}(\mathcal{R})$ and $\mathfrak{U}_{cw}=w\mathfrak{U}(\mathcal{R})$. Now Corollaries \ref{c11} and \ref{c12} directly follows from Theorem \ref{MT1}.

Let $\mathcal{R}(H)\equiv(\emptyset, \{1\})$ and $\mathfrak{F}$ be a hereditary local formation. Again $\mathcal{R}$ is good. Let $\overline{H}\in\mathfrak{F}$ be an $\mathfrak{F}$-eccentric chief factor of an $\mathfrak{F}^*$-group $G$. Note that $r(\overline{H}, G)=1$ and $\overline{H}\rtimes G/C_G(\overline{H})$ is a quotient of $\overline{H}\times\overline{H}\in\mathfrak{F}$. Thus, $\overline{H}$ is $\mathfrak{F}$-central in $G$, a contradiction. It means that $\mathfrak{F}^*=\mathfrak{F}(\mathcal{R})$. Now Corollary \ref{c13} directly follows from Theorem \ref{MT1}.
\end{proof}

\subsection{The structure of an $\mathfrak{F}(\mathcal{R})$-group}

The aim of this subsection is to obtain the characterization of an $\mathfrak{F}(\mathcal{R})$-group.

\begin{definition}
  Let $Z(G, \mathcal{R}, \mathfrak{F}, n)$ be the greatest $G$-invariant subgroup of $G$ such that $\overline{H}\not\in\mathfrak{F}$, $r(\overline{H}, G)>n$  and $gr(\overline{H}, G)\in \mathcal{R}(\overline{H})$   for  every its $G$-composition $\mathfrak{F}$-eccentric in $G$  factor\,$\overline{H}$.
\end{definition}

Let $C$ be a set  and $\mathcal{R}$ be a generalized rank function. We say that $\mathcal{R}(\overline{H})\subseteq C$ if $A_\mathcal{R}(\overline{H})\cup B_\mathcal{R}(\overline{H})\subseteq C$. By $\mathcal{R}(\overline{H})\cap C$ we mean $(A_\mathcal{R}(\overline{H})\cap C, B_\mathcal{R}(\overline{H})\cap C)$.

\begin{remark}
 $(1)$ Let $N$ and $M$ be  normal subgroups of $G$. According to (b) of the Isomorphism Theorems \cite[2.1A]{s8} every $G$-composition factor of $NM$ is $G$-isomorphic to a $G$-composition factor of $N$ or $M$. Hence  $Z(G, \mathcal{R}, \mathfrak{F}, n)$ exists in every group    by Lemma \ref{Lgr}.

 $(2)$ It is clear that $G\in\mathfrak{F}(\mathcal{R})$ iff $G=Z(G, \mathcal{R}, \mathfrak{F}, 0)$.

 $(3)$ If $\mathcal{R}(S)\subseteq[0, 1]$ for every simple group $S$, then  $Z(G, \mathcal{R}, \mathfrak{F}, n)=\mathrm{Z}_\mathfrak{F}(G)$ for $n>1$.
\end{remark}


\begin{theorem}\label{T2}
Let $\mathfrak{F}$ be a   solubly saturated  formation containing all nilpotent groups such that $\mathfrak{F}$ contains every composition factor of every $\mathfrak{F}$-group and $\mathcal{R}$ be a generalized rank function.   Then
the following statements are equivalent:

$(1)$ $G$ is an $\mathfrak{F}(\mathcal{R})$-group.

$(2)$
Let $Z=Z(G, \mathcal{R}, \mathfrak{F}, 4)$. Then $gr(N/Z, G)\in \mathcal{R}(N/Z)\cap [1, 4]$ for every minimal normal subgroup $N/Z$ of $G/Z$   and   $(G/Z)/\mathrm{Soc}(G/Z)$ is a soluble $\mathfrak{F}$-group.

$(3)$ The following holds:

\quad $(a)$ $G^\mathfrak{F}=G^{E\mathfrak{F}}$.

\quad $(b)$ If $N\trianglelefteq G$ and $N\leq G^\mathfrak{F}$, then $(G^\mathfrak{F}/N)_{E\mathfrak{F}}=\mathrm{Z}(G^\mathfrak{F}/N)$.

\quad $(c)$ Let $n$ be the least number such that there is a simple  non-$\mathfrak{F}$-section in $S_{n+1}$ and  $T=G^\mathfrak{F}\cap Z(G, \mathcal{R}, \mathfrak{F}, n)$.  Then $G^\mathfrak{F}/T\leq \mathrm{Soc}(G/T)$ and $N/T\not\in\mathfrak{F}$, $r(N/T, G)\leq n$ and $gr(N/T, G)\in\mathcal{R}(N/T)$ for every minimal normal subgroup $N/T$ of $G/T$ from $G^\mathfrak{F}/T$.
\end{theorem}

Recall \cite[p. 13]{s5} that a group is called \emph{semisimple} provided it is either identity or the direct product of some simple non-abelian groups.

\begin{corollary}[{\cite[X, 13.6]{s5}}]\label{c21}
  A group $G$ is quasinilpotent if and only if $G/\mathrm{Z}_\infty(G)$ is semi\-simple.
\end{corollary}

\begin{corollary}[{\cite[Theorem 2.8]{sk2}}]\label{c22}
  Let $\mathfrak{F}$ be a normally hereditary saturated formation containing all nilpotent groups. A group $G$ is a quasi-$\mathfrak{F}$-group  if and only if $G/\mathrm{Z}_\mathfrak{F}(G)$ is semisimple.
\end{corollary}

\begin{corollary}[{\cite[Theorem A]{j7}}]\label{c23}
 Let $\mathfrak{N}\subseteq\mathfrak{F}$ be a saturated formation of soluble groups. Then $G\in\mathfrak{F}_{ca}$ if and only if $G^\mathfrak{F}=G^\mathfrak{S}$, $\mathrm{Z}(G^\mathfrak{F})\leq\mathrm{Z}_\mathfrak{F}(G)$ and $G^\mathfrak{F}/\mathrm{Z}(G^\mathfrak{F})$ is a direct product of $G$-invariant simple non-abelian groups.
\end{corollary}

\begin{corollary}[{\cite[Proposition 2.4]{j1}}]\label{c24}
  A group $G$ is  $c$-supersoluble  if and only if there is a perfect
normal subgroup $D$ such that $G/D$ is supersoluble, $D/\mathrm{Z}(D)$ is a direct product of
$G$-invariant simple groups, and $\mathrm{Z}(D)$ is supersolubly embedded in $G$.
\end{corollary}

\begin{corollary}[{\cite[Theorem B]{j4}}]\label{c25}
  A group $G$ is widely $c$-supersoluble if and only if   $G^{w\mathfrak{U}}=G^\mathfrak{S}$, $\mathrm{Z}(G^{w\mathfrak{U}})\leq\mathrm{Z}_\mathfrak{F}(G)$ and $G^{w\mathfrak{U}}/\mathrm{Z}(G^{w\mathfrak{U}})$ is a direct product of $G$-invariant simple non-abelian groups.
\end{corollary}

\begin{proof}[{\bf Proof of Theorem \ref{T2}}]

$(1)\Rightarrow(2)$ Let $G\in\mathfrak{F}(\mathcal{R})$, $Z=Z(G, \mathcal{R}, \mathfrak{F}, 4)$ and $K/Z=\overline{K}=\mathrm{Soc}(G/Z)$.   From the definition of $\mathfrak{F}(\mathcal{R})$ it follows  that $\overline{G}=G/Z$ does not have minimal normal   $\mathfrak{F}$-subgroups. Note that $r(\overline{K}_i, \overline{G})\leq 4$ for every minimal normal subgroup $\overline{K}_i$  of $\overline{G}$ $(i=1,\dots, n)$ by the definition of $Z(G, \mathcal{R}, \mathfrak{F}, 4)$. Note that $\overline{K}_i=\overline{K}_{i,1}\times\dots\times \overline{K}_{i,k}$ is the direct product of isomorphic simple groups and $1\leq k\leq 4$. Hence $\mathrm{Aut}(\overline{K}_i)\simeq \mathrm{Aut}(\overline{K}_{i,1})\wr S_k$ by \cite[1.1.20]{s9}. Note that $S_k$ is soluble and $\mathrm{Out}(\overline{K}_{i,1})$ is soluble by Schreier conjecture. It means that $\mathrm{Out}(\overline{K}_i)$ is soluble.

  Since $\mathfrak{N}\subseteq \mathfrak{F}$, $\Phi(\overline{G})\simeq 1$. Hence $\overline{K}=\tilde{\mathrm{F}}(\overline{G})$. Recall that $\overline{K}=\overline{K}_1\times\dots\times \overline{K}_n$.
Note that every element $xZ$ induces an automorphism $\alpha_{x, i}$ on $\overline{K}_i$ for $i=1,\dots, n$. Let $$\varphi: xZ\to(\alpha_{x, 1},\dots, \alpha_{x, n})$$
It is clear that $\varphi(xZ)\varphi(yZ)=\varphi(xyZ)$. Also note that if $\varphi(xZ)=\varphi(yZ)$, then $y^{-1}xZ$ acts trivially on every $\overline{K}_i$. According to  \cite[\S7, 7.11]{d8}
 $C_G(\mathrm{\tilde{F}}(G))\subseteq\mathrm{\tilde{F}}(G)$. So  $$y^{-1}xZ\in \cap_{i=1}^nC_{\overline{G}}(\overline{K}_i)=C_{\overline{G}}(\overline{K})=
C_{\overline{G}}(\tilde{\mathrm{F}}(\overline{G}))\subseteq \tilde{\mathrm{F}}(\overline{G})=\overline{K}. $$
Hence $y^{-1}xZ=1Z$. Now $yZ=xZ$ and $\varphi$ is injective. Hence $\varphi$ is the monomorphism from $\overline{G}$ to $\mathrm{Aut}(\overline{K}_1)\times\dots\times\mathrm{Aut}(\overline{K}_n)$.  Note that $\varphi(\overline{K})=\mathrm{Inn}(\overline{K}_1)\times\dots\times\mathrm{Inn}(\overline{K}_n)$. It is straightforward to check that
$$(\mathrm{Aut}(\overline{K}_1)\times\dots\times\mathrm{Aut}(\overline{K}_n))
/(\mathrm{Inn}(\overline{K}_1)\times\dots\times\mathrm{Inn}(\overline{K}_n))
\simeq\mathrm{Out}(\overline{K}_1)\times\dots\times\mathrm{Out}(\overline{K}_n)$$
Now $G/K\simeq\overline{G}/\overline{K}\simeq\varphi(\overline{G})/\varphi(\overline{K}) $ can be viewed as subgroup of    $\mathrm{Out}(\overline{K}_1)\times\dots\times\mathrm{Out}(\overline{K}_n)\in\mathfrak{S}$. Hence every chief factor of $G$ above $K$ is soluble and, hence, $\mathfrak{F}$-central in $G$.

  $(2), (3)\Rightarrow(1)$ From $(2)$ or $(c)$ of $(3)$ it follows that
  a group $G$ has a chief series such that  $\overline{H}\not\in\mathfrak{F}$ and $gr(\overline{H}, G)\in\mathcal{R}(\overline{H})$ for  every its $\mathfrak{F}$-eccentric chief factor $\overline{H}$.
  By Jordan-H\"{o}lder theorem and Lemma \ref{Lgr} it follows that every chief series of $G$ has this property. Thus $G\in\mathfrak{F}$.

   $(1)\Rightarrow(3)$  Assume now that $G\in\mathfrak{F}(\mathcal{R})$. $(a)$ Note that every chief factor of $G$ above $G^{E\mathfrak{F}}$ is an $\mathfrak{F}$-group and hence $\mathfrak{F}$-central in $G$ by the definition of $\mathfrak{F}(\mathcal{R})$. So $\mathrm{Z}_\mathfrak{F}(G/G^{E\mathfrak{F}})=G/G^{E\mathfrak{F}}$. Therefore $G/G^{E\mathfrak{F}}\in\mathfrak{F}$ and $G^\mathfrak{F}\leq G^{E\mathfrak{F}}$. From $\mathfrak{F}\subseteq{E\mathfrak{F}}$ it follows that $G^{E\mathfrak{F}}\leq G^{\mathfrak{F}}$. Thus $G^\mathfrak{F}= G^{E\mathfrak{F}}$.

$(b)$ Note that $H_{E\mathfrak{F}}\leq \mathrm{Z}_\mathfrak{F}(H)$ for every $\mathfrak{F}(\mathcal{R})$-group $H$.   By \cite[Corollary 2.3.1]{z}, $G^\mathfrak{F}\leq C_G(\mathrm{Z}_\mathfrak{F}(G))$. Hence $G^\mathfrak{F}\cap\mathrm{Z}_\mathfrak{F}(G)=\mathrm{Z}(G^\mathfrak{F})$. So if $N\trianglelefteq G$ and $N\leq G^\mathfrak{F}$, then $(G^\mathfrak{F}/N)_{E\mathfrak{F}}=\mathrm{Z}(G^\mathfrak{F}/N)$.

$(c)$ Let $n$ be the least number such that there is a simple  non-$\mathfrak{F}$-group in $S_{n+1}$ and  $T=G^\mathfrak{F}\cap Z(G, \mathcal{R}, \mathfrak{F}, n)$. Let $N/T=\overline{N}$ be a minimal normal subgroup  of $\overline{G}=G/T$ that lies in $\overline{G^\mathfrak{F}}=G^\mathfrak{F}/T$.  From the definition of $Z(G, \mathcal{R}, \mathfrak{F}, n)$ it follows that $\overline{N}\not\in\mathfrak{F}$ and $r(\overline{N}, G)\leq n$. Note that $(\overline{G^\mathfrak{F}})_\mathfrak{S}\simeq 1 $ by $\mathfrak{N}\subseteq\mathfrak{F}$. Hence $\Phi(\overline{G^\mathfrak{F}})\simeq 1$. From Lemma \ref{l3} it follows that $\mathrm{Soc}(\overline{G})\cap \overline{G^\mathfrak{F}}=\mathrm{Soc}(\overline{G^\mathfrak{F}})=\mathrm{\tilde{F}}(\overline{G^\mathfrak{F}})$.
Note that $\overline{N}=\overline{N}_1\times\dots\times\overline{N}_k$ is a direct product of isomorphic simple non-abelian groups. Recall that $k\leq n$.

By \cite[1.1.20]{s9} $\mathrm{Aut}(\overline{N})\simeq \mathrm{Aut}(\overline{N}_1)\wr S_k$.  Now every subgroup of $\mathrm{Aut}(\overline{N})/\overline{N}$ belongs $  E\mathfrak{F}$ by the definition of $n$,   Schreier conjecture and $\mathfrak{S}\subseteq E\mathfrak{F}$.    The same arguments as in $(1)\Rightarrow(2)$ shows that $\overline{G^\mathfrak{F}}/\mathrm{\tilde{F}}(\overline{G^\mathfrak{F}})\in E\mathfrak{F}$.
From $G^\mathfrak{F}=G^{E\mathfrak{F}}$ it follows that  $\overline{G^\mathfrak{F}}/\mathrm{\tilde{F}}(\overline{G^\mathfrak{F}})=
\overline{G^\mathfrak{F}}/(\mathrm{Soc}(\overline{G})\cap \overline{G^\mathfrak{F}})\simeq 1 $.
    \end{proof}

\begin{proof}[{\bf Proof of Corollaries \ref{c21}, \ref{c22}, \ref{c23}, \ref{c24} and \ref{c25}}]

Note that Corollary \ref{c21} directly follows from Corollary \ref{c22}.

Let $\mathcal{R}(H)\equiv(\emptyset, \{1\})$ and $G\in\mathfrak{F}^*$. Then $\mathfrak{F}(\mathcal{R})=\mathfrak{F}^*$ and $Z=Z(G, \mathcal{R}, \mathfrak{F}, 4)=\mathrm{Z}_\mathfrak{F}(G)$.
According to $(1)$ of the proof of Theorem \ref{T2} $(G/Z)/\mathrm{Soc}(G/Z)$ is isomorphic to a subgroup of the outer automorphisms  group of  $\mathrm{Soc}(G/Z)$ induced by $G$. Note that in this case $G$ induces  inner automorphism groups on every minimal normal subgroup of $G/Z$. It means that $G/Z\simeq 1$. Thus Corollary \ref{c22} is proved.

  Let $\mathfrak{F}$ be a hereditary saturated formation,  $\mathfrak{N}\subseteq\mathfrak{F}\subseteq\mathfrak{S}$ and $\mathcal{R}(H)\equiv(\{1\},\emptyset)$. Then $\mathfrak{F}(\mathcal{R})=\mathfrak{F}_{ca}$. Note that $E\mathfrak{F}=\mathfrak{S}$.  Let $n$ be the least number such that there is a simple  non-$\mathfrak{F}$-section in $S_{n+1}$. It is clear that $n=4$. So $Z(G, R, \mathfrak{F}, n)=\mathrm{Z}_\mathfrak{F}(G)$. Now $G^\mathfrak{F}\cap Z(G, R, \mathfrak{F}, n)=G^\mathfrak{F}\cap\mathrm{Z}_\mathfrak{F}(G)\leq (G^\mathfrak{F})_{E\mathfrak{F}}$. Thus Corollary \ref{c23} directly follows from $(3)$ of Theorem \ref{T2}.

   Note that the class of all widely $c$-supersoluble groups $\mathfrak{U}_{cw}=(w\mathfrak{U})_{ca}$.
  So Corollary \ref{c25} directly follows from Corollary \ref{c23}.

  Recall that a normal subgroup $H$ of a group $G$ is supersoluble embedded  iff $H\leq \mathrm{Z}_\mathfrak{U}(G)$; $E\mathfrak{U}=\mathfrak{S}$ and $(G^\mathfrak{S})'=G^\mathfrak{S}$. Now  Corollary \ref{c24} directly follows from Corollary \ref{c23}.
\end{proof}

\subsection{On one question of L.\,A. Shemetkov}

The main result of this section that makes the contribution to the solution of Question \ref{q1} is:

\begin{theorem}\label{T3}
  Let $\mathfrak{F}$ be a hereditary saturated formation containing all nilpotent groups, $m$ be a natural number with $\mathfrak{G}_{\{q\in\mathbb{P}\,|\,q\leq m\}}\subseteq\mathfrak{F}$,     $\mathcal{R}$ be a very good generalized rank function such that $\mathcal{R}(N)\subseteq[0, m]$ for any simple group $N$. Then the following statements are equivalent:

  $(1)$ $\mathrm{Z}_{\mathfrak{F}}(G)= \mathrm{Int}_{\mathfrak{F}}(G)$ holds for every group $G$ and
   \begin{center} $\bigcup_{n=1}^m(\mathrm{Out}(G)\wr S_n\,|\,G\not\in\mathfrak{F}$ is a simple group and $n\in A_\mathcal{R}(G))\subseteq\mathfrak{F}$.
   \end{center}

  $(2)$ $\mathrm{Z}_{\mathfrak{F}(\mathcal{R})}(G)= \mathrm{Int}_{\mathfrak{F}(\mathcal{R})}(G)$ holds for every group $G$.
\end{theorem}

\begin{corollary}[{\cite[Theorem 1]{ArX}}]\label{c30}
Let $\mathfrak{F}$ be a hereditary saturated formation containing all nilpotent groups. Then   $\mathrm{Int}_\mathfrak{F}(G)=\mathrm{Z}_\mathfrak{F}(G)$ holds for every group $G$  if and only if  $\mathrm{Int}_{\mathfrak{F}^*}(G)=\mathrm{Z}_{\mathfrak{F}^*}(G)$ holds for every group $G$.
\end{corollary}

\begin{corollary}\label{c31}
  Let $\mathcal{R}$ be a very good generalized rank function. Then  $\mathrm{Z}_{\mathfrak{N}(\mathcal{R})}(G)= \mathrm{Int}_{\mathfrak{N}(\mathcal{R})}(G)$ holds for every group $G$ if and only if for any simple non-abelian group $N$ holds:

   $(1)$  $\mathcal{R}(N)\subseteq[0, 2]$;

   $(2)$  if $1\in A_\mathcal{R}(N)$, then $\mathrm{Out}(N)$ is nilpotent;

   $(3)$ if $2\in A_\mathcal{R}(N)$, then $\mathrm{Out}(N)$ is a $2$-group.
\end{corollary}

\begin{example}
  Let $\mathfrak{F}_1$ (resp. $\mathfrak{F}_2$) be a class of groups whose abelian chief factors are central and non-abelian chief factors are arbitrary (resp. are directs products of at most 2 alternating groups). Then $\mathrm{Z}_{\mathfrak{N}^*}(G)= \mathrm{Int}_{\mathfrak{N}^*}(G)$ and $\mathrm{Z}_{\mathfrak{F}_2}(G)= \mathrm{Int}_{\mathfrak{F}_2}(G)$  hold for every group $G$ and there exist groups $G_1$ and $G_2$ with   $\mathrm{Z}_{\mathfrak{N}_{ca}}(G_1)\neq \mathrm{Int}_{\mathfrak{N}_{ca}}(G_1)$ and $\mathrm{Z}_{\mathfrak{F}_1}(G_2)\neq \mathrm{Int}_{\mathfrak{F}_1}(G_2)$.
\end{example}

It is important to mention that if $\mathrm{Z}_{\mathfrak{F}(\mathcal{R})}(G)= \mathrm{Int}_{\mathfrak{F}(\mathcal{R})}(G)$ holds for every group $G$, then $\mathcal{R}$ is bounded:

\begin{theorem}\label{t31}
   Let $\mathfrak{F}\neq\mathfrak{G}$ be a hereditary saturated formation containing all nilpotent groups and $\mathcal{R}$ be a very good generalized rank function.

   $(1)$ Assume that    $\mathrm{Z}_{\mathfrak{F}(\mathcal{R})}(G)= \mathrm{Int}_{\mathfrak{F}(\mathcal{R})}(G)$ holds for every group $G$. Let
   $$ C_1=\min_{G\in\mathcal{M}(\mathfrak{F})\textrm{ with }\mathrm{F}(G)=\mathrm{\tilde{F}}(G)}\max_{M \textrm{ is a maximal subgroup of }G}|M|-1.$$ Then $\mathcal{R}(S)\subseteq [0, C_1]$ for every simple group $S\not\in\mathfrak{F}$.

   $(2)$ Let
   $$C_2=\max{\{m\in\mathbb{N}\,|\,\mathfrak{G}_{\{q\in\mathbb{P}\,|\,q\leq m\}}\subseteq\mathfrak{F}\}}. $$
    If $\mathcal{R}(S)\subseteq [0, C_2]$ for every simple group $S\not\in\mathfrak{F}$, then $gr(\overline{H}, G)\in \mathcal{R}(\overline{H})$ for every $G$-composition factor $\overline{H}\not\in\mathfrak{F}$
    below $\mathrm{Int}_{\mathfrak{F}(\mathcal{R})}(G)$.
\end{theorem}

Recall \cite[3.4.5]{s9} that every  solubly saturated  formation $\mathfrak{F}$ contains the greatest saturated subformation $\mathfrak{F}_l$ with respect to set inclusion.

\begin{theorem}[{\cite[3.4.5]{s9}}]\label{t4}
       Let  $F$ be the canonical composition definition of a non-empty solubly saturated formation $\mathfrak{F}$. Then $f$ is a  local definition of   $\mathfrak{F}_l$, where $f(p)=F(p)$ for all $p\in\mathbb{P}$.\end{theorem}

The following result directly follows from   Theorems \ref{MT1} and \ref{t4}.

\begin{proposition}\label{t32}
       Let  $\mathfrak{F}$ be a local formation containing all nilpotent groups and $\mathcal{R}$ be a generalized rank function. Then $\mathfrak{F}(\mathcal{R})_l=\mathfrak{F}$.\end{proposition}

In the view of Theorem \ref{t4} A.F. Vasil'ev asked if it is possible to reduce Question \ref{q1} for solubly saturated formations to the case of saturated formations.
Recall that $D_0\mathfrak{X}$ is the class of groups which are the direct products $\mathfrak{X}$-groups. Partial answer on A.F. Vasil'ev's   question is given in

  \begin{theorem}\label{t3}
Let  $F$ be the canonical composition definition of   a non-empty solubly saturated formation $\mathfrak{F}$. Assume that  $F(p)\subseteq\mathfrak{F}_l$ for all $p\in\mathbb{P}$ and   $\mathfrak{F}_l$ is hereditary.


$(1)$ Assume that $\mathrm{Int}_{\mathfrak{F}_l}(G)=\mathrm{Z}_{\mathfrak{F}_l}(G)$ holds for every group $G$. Let \vspace{-3mm}
$$\mathfrak{H}=(S\textrm{ is a simple group }| \textrm{ every } \mathfrak{F}\textrm{-central chief }  D_0(S)\textrm{-factor  is } \mathfrak{F}_l\textrm{-central}). \vspace{-3mm}$$
Then every   chief $D_0\mathfrak{H}$-factor of $G$ below    $\mathrm{Int}_{\mathfrak{F}}(G)$ is $\mathfrak{F}_l$-central in $G$.

$(2)$ \cite[Theorem 3]{ArX} If $\mathrm{Int}_\mathfrak{F}(G)=\mathrm{Z}_\mathfrak{F}(G)$ holds for every group $G$, then $\mathrm{Int}_{\mathfrak{F}_l}(G)=\mathrm{Z}_{\mathfrak{F}_l}(G)$ holds for every group $G$.
\end{theorem}

\begin{proof}[\bf{Proof of Theorem \ref{t3}}]
 From $F(p)=\mathfrak{N}_pF(p)\subseteq\mathfrak{F}_l$  and Theorem
\ref{t4} it follows that if we restrict $F$ to $\mathbb{P}$, then we obtain the canonical local definition of $\mathfrak{F}_l$.

$(1)$  Let    $\overline{H}=H/K$ be a  chief $D_0\mathfrak{H}$-factor of $G$ below $\mathrm{Int}_\mathfrak{F}(G)$.

$(a)$ \emph{If $\overline{H}$ is abelian, then $MC_G(\overline{H})/C_G(\overline{H}) \in F(p)$ for every $\mathfrak{F}$-maximal subgroup $M$ of $G$.}

If $\overline{H}$ is abelian, then it is an elementary abelian $p$-group for some  $p$ and $\overline{H}\in\mathfrak{F}$. Let $M$ be an $\mathfrak{F}$-maximal subgroup of $G$ and $K=H_0\trianglelefteq H_1\trianglelefteq\dots\trianglelefteq H_n=H$ be a part of a chief series of $M$. Note that $H_i/H_{i-1}$ is an $\mathfrak{F}$-central chief factor of $M$ for all $i=1,\dots, n$. So $M/C_M(H_i/H_{i-1})\in F(p)$ by Lemma \ref{l1} for all $i=1,\dots, n$.
Therefore $M/C_M(\overline{H})\in\mathfrak{N}_pF(p)=F(p)$ by \cite[Lemma 1]{j3}. Now $$MC_G(\overline{H})/C_G(\overline{H})\simeq M/C_M(\overline{H})\in F(p)$$ for every $\mathfrak{F}$-maximal subgroup $M$ of $G$.

$(b)$ \emph{If $\overline{H}$ is non-abelian, then $MC_G(\overline{H})/C_G(\overline{H}) \in F(p)$ for every $\mathfrak{F}$-maximal\,subgroup $M$\,of\,$G$.}

If $\overline{H}$ is non-abelian, then
 it is a direct product of isomorphic non-abelian simple groups. Let $M$ be an $\mathfrak{F}$-maximal subgroup of $G$. By Lemma \ref{l3}, $\overline{H}=\overline{H}_1\times\dots\times \overline{H}_n$ is a direct product of minimal normal subgroups $\overline{H}_i$  of $\overline{M}=M/K$.
Now $\overline{H}_i$ is $\mathfrak{F}$-central in $\overline{M}$ for all $i=1,\dots, n$. Hence $\overline{H}_i$ is $\mathfrak{F}_l$-central in $\overline{M}$ for all $i=1,\dots, n$ by the definition of $\mathfrak{H}$.  Therefore $M/C_M(\overline{H}_i)\in F(p)$ for all $p\in\pi(\overline{H}_i)$ by Lemma \ref{l1}. Note that $C_M(\overline{H})=\cap_{i=1}^n C_M(\overline{H}_i)$. Since $F(p)$ is a formation,    $$M/\cap_{i=1}^n C_M(\overline{H}_i)= M/C_M(\overline{H})\in F(p)$$ for all $p\in\pi(\overline{H})$. It means that $MC_G(\overline{H})/C_G(\overline{H})\simeq M/C_M(\overline{H})\in F(p)$ for every $p\in\pi(\overline{H})$ and every $\mathfrak{F}$-maximal subgroup $M$ of $G$.

$(c)$ \emph{All $\mathfrak{F}_l$-subgroups of $G/C_G(\overline{H})$ are $F(p)$-groups for all $p\in\pi(\overline{H})$.}

Let $Q/C_G(\overline{H})$ be an  $\mathfrak{F}_l$-maximal subgroup of $G/C_G(\overline{H})$. Then there exists an  $\mathfrak{F}_l$-maximal subgroup $N$ of $G$ with $NC_G(\overline{H})/C_G(\overline{H})=Q/C_G(\overline{H})$ by \cite[1,  5.7]{s5}. From $\mathfrak{F}_l\subseteq\mathfrak{F}$ it follows that there exists an $\mathfrak{F}$-maximal subgroup $L$ of $G$ with $N\leq L$. So  $$Q/C_G(\overline{H})\leq LC_G(\overline{H})/C_G(\overline{H})\in F(p) \textrm{ for all }p\in\pi(\overline{H})$$ by $(a)$ and $(b)$. Since $F(p)$ is hereditary by \cite[IV, 3.16]{s8},  $Q/C_G(\overline{H})\in F(p)$. It means that all $\mathfrak{F}_l$-maximal subgroups of $G/C_G(\overline{H})$ are $F(p)$-groups. Hence all $\mathfrak{F}_l$-subgroups of $G/C_G(\overline{H})$ are $F(p)$-groups.

$(d)$ \emph{$\overline{H}$  is $\mathfrak{F}_l$-central in $G$.}

Assume now that $\overline{H}$ is not  $\mathfrak{F}_l$-central in $G$. So $G/C_G(\overline{H})\not\in F(p)$ for some $p\in\pi(\overline{H})$ by Lemma \ref{l1}. It means that    $G/C_G(\overline{H})$ contains an $s$-critical for $F(p)$ subgroup  $S/C_G(\overline{H})$. Since   $\mathrm{Int}_{\mathfrak{F}_l}(G)=\mathrm{Z}_{\mathfrak{F}_l}(G)$ holds for every group $G$,    $S/C_G(\overline{H})\in\mathfrak{F}_l$ by \cite[Theorem A]{h4}. Therefore $S/C_G(\overline{H})\in F(p)$ by $(c)$, a contradiction. Thus $\overline{H}$ is    $\mathfrak{F}_l$-central in $G$.
\end{proof}

\begin{proof}[\bf{Proof of Theorem \ref{t31}}]
$(1)$ Since $\mathrm{Z}_{\mathfrak{F}(\mathcal{R})}(G)= \mathrm{Int}_{\mathfrak{F}(\mathcal{R})}(G)$ holds for every group $G$, we see that $\mathrm{Z}_{\mathfrak{F}}(G)= \mathrm{Int}_{\mathfrak{F}}(G)$ holds for every group $G$ by Theorem \ref{t31} and Proposition \ref{t32}.

$(a)$ \emph{There is $G\in\mathcal{M}(\mathfrak{F})$ with $\mathrm{F}(G)=\mathrm{\tilde{F}}(G)$. }

 From $\mathfrak{F}\neq\mathfrak{G}$ it follows that there exist $F(p)$-critical groups for some $p$. Let   $N$ be the minimal order group among them. Then $\mathrm{O}_p(N)=1$ and $N$ has the unique minimal normal subgroup. Note that $N\in\mathfrak{F}$ by \cite[Theorem A]{h4}. There exists a simple $\mathbb{F}_pN$-module $M$ which is faithful for $N$ by \cite[10.3B]{s8}. Let     $G=M\rtimes N$. Note that  $M=\mathrm{F}(G)=\mathrm{\tilde{F}}(G)$ and $G\not\in\mathfrak{F}$. From $F(p)=\mathfrak{N}_pF(p)$ it follows that $G\in\mathcal{M}(\mathfrak{F})$.

$(b)$ \emph{If $G\in\mathcal{M}(\mathfrak{F})$ with $\mathrm{F}(G)=\mathrm{\tilde{F}}(G)$, then $G\not\in\mathfrak{F}(\mathcal{R})$ for any generalized rank function $\mathcal{R}$.          }

Since $\mathfrak{F}$ is saturated, we see that $G/\Phi(G)\in\mathcal{M}(\mathfrak{F})$.  Note that $G/\Phi(G)$ has a unique minimal normal subgroup $N/\Phi(G)$ and $N/\Phi(G)$ is an abelian $\mathfrak{F}$-eccentric chief factor of $G$. From $\mathfrak{N}\subseteq\mathfrak{F}$ and the definition of $\mathfrak{F}(\mathcal{R})$ it follows that     $G\not\in\mathfrak{F}(\mathcal{R})$ for any generalized rank function $\mathcal{R}$.

$(c)$ \emph{ Let $G\in\mathcal{M}(\mathfrak{F})$ with $\mathrm{F}(G)=\mathrm{\tilde{F}}(G)$, $m$ be the greatest number\,among\,orders\,of\,maximal subgroups\,of\,$G$,\, $S$\,be\,a\,simple\,group\,with\,$m\in\mathcal{R}(S)$\,and\,$T=S\wr_{reg} G$.\,Then\,  $\mathrm{Z}_{\mathfrak{F}(\mathcal{R})}(T)\neq \mathrm{Int}_{\mathfrak{F}(\mathcal{R})}(T)$.}

Let $N$ be the base of $T$ and $M$ be an $\mathfrak{F}(\mathcal{R})$-maximal subgroup of $T$. Note that $MN/N\in \mathfrak{F}(\mathcal{R})$ is isomorphic to some subgroup of $G\not\in\mathfrak{F}(\mathcal{R})$. Let $K$ be a maximal subgroup of $G$. Then $K\in\mathfrak{F}$. Let show that $NK\in\mathfrak{F}(\mathcal{R})$. From the properties of wreath product it follows that $N$ is the direct product of $|G:K|$ minimal normal subgroups of rank $|K|$ of $NK$ and if some element on $NK$ fixes a composition factor of some of these subgroups, then it induces an inner automorphism on it.
It means that $NK\in\mathfrak{F}(\mathcal{R})$. Hence
$N\leq \mathrm{Int}_{\mathfrak{F}(\mathcal{R})}(T)$.
Assume that  $N\leq \mathrm{Z}_{\mathfrak{F}(\mathcal{R})}(T)$. So $T/C_T(N)\in\mathfrak{F}(\mathcal{R})$. Then $G\not\in\mathfrak{F}(\mathcal{R})$ is isomorphic to a quotient of $T/C_T(N)$, a contradiction. Hence $N$ is an $\mathfrak{F}(\mathcal{R})$-eccentric chief factor of $T$.

$(d)$ \emph{The final step.}

From $(c)$ it follows that if
$$ \min_{G\in\mathcal{M}(\mathfrak{F})\textrm{ with }\mathrm{F}(G)=\mathrm{\tilde{F}}(G)}\max_{M \textrm{ is a maximal subgroup of }G}|M|\in\mathcal{R}(S)\textrm{ for some simple group }S\not\in\mathfrak{F},$$
then there is a group $T$  with  $\mathrm{Z}_{\mathfrak{F}(\mathcal{R})}(T)\neq \mathrm{Int}_{\mathfrak{F}(\mathcal{R})}(T)$. The contradiction to $\mathrm{Z}_{\mathfrak{F}(\mathcal{R})}(G)= \mathrm{Int}_{\mathfrak{F}(\mathcal{R})}(G)$ holds for every group $G$.

$(2)$ Let $\pi={\{q\in\mathbb{P}\,|\,q\leq C_2\}}$ and  $\overline{H}=H/K\not\in\mathfrak{F}$ be a
$G$-composition factor     below $\mathrm{Int}_{\mathfrak{F}(\mathcal{R})}(G)$. Then $\overline{H}=\overline{H}_1\times\dots\times\overline{H}_n$ is the direct product of  isomorphic simple groups. Let $T/K=\overline{T}=\cap_{i=1}^n N_{\overline{G}}(\overline{H}_i)$.

Let $M$ be an $\mathfrak{F}(\mathcal{R})$-maximal subgroup of $G$. Then $H/K\leq M/K=\overline{M}\in\mathfrak{F}(\mathcal{R})$. Since $\overline{H}$ is non-abelian it is the direct product of minimal normal subgroups of $M$ by Lemma \ref{l3}. Let $\overline{N}_i=\overline{N}_{i,1}\times\dots\times\overline{N}_{i, k}$ be one of them and $\overline{T}_i=\cap_{j=1}^k N_{\overline{M}}(\overline{N}_{i, j})$. Then $\overline{M}/\overline{T}_i$ is isomorphic to a subgroup of the symmetric group of degree $k$ by \cite[1.1.40(6)]{s9}. From $k\leq C_2$ it follows that $\overline{M}/\overline{T}_i\in\mathfrak{G}_\pi$. Since $\mathfrak{G}_\pi$ is a formation, we see that  $\overline{M}\overline{T}/\overline{T}\simeq \overline{M}/\cap_i\overline{T}_i\in\mathfrak{G}_\pi$.

From $\mathfrak{N}\subseteq\mathfrak{F}$ it follows that every element of $G$ lies in some $\mathfrak{F}(\mathcal{R})$-maximal subgroup of $G$. Hence   $\overline{G}/\overline{T}\simeq G/T\in\mathfrak{G}_\pi$ and this group is isomorphic to a transitive group of permutations on $\{\overline{H}_1,\dots,\overline{H}_n\}$. According to \cite[1,  5.7]{s5} there is a $\pi$-subgroup $L$ of $G$ with $LT=G$. Note that $L\in\mathfrak{F}$. Hence there is an $\mathfrak{F}(\mathcal{R})$-maximal subgroup $Q$ of $G$ with $L\leq Q$. So     $LH\leq Q$.
Now $H/K$ is a minimal normal subgroup of $Q/K$. Therefore $m=r(\overline{H}, G)=gr(\overline{H}, Q)\in\mathcal{R}(\overline{H})$.

If $m\in A_\mathcal{R}(\overline{H})$, then we are done. Assume that $m\in B_\mathcal{R}(\overline{H})$. Now for every $x\in G$, there is  a $\mathfrak{F}(\mathcal{R})$-maximal subgroup $P$ of $G$ with $x\in Q$. So $H/K$ is a direct product of minimal normal subgroups of $Q/K$. Since $\mathcal{R}$ is very good generalized rank function,
if $x$ fixes a composition factor of $\overline{H}$, then it induces an inner automorphism on it.  Thus $gr(\overline{H}, G)\in\mathcal{R}(\overline{H})$.
\end{proof}

\begin{proposition}\label{p7}
  Let $\mathfrak{F}$ be a hereditary saturated formation containing all nilpotent groups and $\mathcal{R}$ be a very good generalized rank function. Then $\mathrm{Z}_{\mathfrak{F}(\mathcal{R})}(G)\leq \mathrm{Int}_{\mathfrak{F}(\mathcal{R})}(G)$.
\end{proposition}

\begin{proof} Let $\mathfrak{F}$ be a hereditary saturated formation with the canonical local definition $F$,
$M$ be an    $\mathfrak{F}(\mathcal{R})$-maximal subgroup of $G$ and $N=M\mathrm{Z}_{\mathfrak{F}(\mathcal{R})}(G)$. Let show that $N\in \mathfrak{F}(\mathcal{R})$. It is sufficient to show that $H/K\not\in\mathfrak{F}$ and $gr(H/K, G)\in\mathcal{R}(H/K)$  for every  $\mathfrak{F}$-eccentric chief factor $H/K$ of $N$ below $\mathrm{Z}_{\mathfrak{F}(\mathcal{R})}(G)$.

 Let $1=Z_0\trianglelefteq Z_1\trianglelefteq\dots\trianglelefteq Z_n=\mathrm{Z}_{\mathfrak{F}(\mathcal{R})}(G)$ be  a chief series of $G$ below $\mathrm{Z}_{\mathfrak{F}(\mathcal{R})}(G)$. Then we may assume that $Z_{i-1}\leq K\leq H\leq Z_i$ for some $i$ by the Jordan-H\"{o}lder theorem.
Note that $$ (Z_i/Z_{i-1})\rtimes G/C_G(Z_i/Z_{i-1})\in\mathfrak{F}(\mathcal{R}).$$
Hence if $Z_i/Z_{i-1}\not\in\mathfrak{F}$, then $gr(Z_i/Z_{i-1}, G)\in\mathcal{R}(Z_i/Z_{i-1})$ and $Z_i/Z_{i-1}$ is non-abelian. Note that every composition factor of $H/K$ is a composition factor of $Z_i/Z_{i-1}$ and $\mathcal{R}$ is a very good generalized rank function. Hence $gr(H/K, N)\in\mathcal{R}(H/K)$.

If $Z_i/Z_{i-1}\in\mathfrak{F}$, then it is   an $\mathfrak{F}$-central chief factor of $(Z_i/Z_{i-1})\rtimes G/C_G(Z_i/Z_{i-1})$. In this case $G/C_G(Z_i/Z_{i-1})\in F(p)$ for all $p\in\pi(Z_i/Z_{i-1})$ by Lemma \ref{l1}. Since $F(p)$ is hereditary by \cite[IV, 3.16]{s8} $$NC_G(Z_i/Z_{i-1})/C_G(Z_i/Z_{i-1})\simeq N/C_N(Z_i/Z_{i-1})\in F(p)$$ for all $p\in\pi(Z_i/Z_{i-1})$. Note that $N/C_N(H/K)$  is a quotient group of $N/C_N(Z_i/Z_{i-1})$. Thus $H/K$ is  an $\mathfrak{F}$-central chief factor of $N$  by Lemma \ref{l1}.

Hence $N\in\mathfrak{F}(\mathcal{R})$. So $N=M\mathrm{Z}_{\mathfrak{F}(\mathcal{R})}(G)=M$. Therefore $\mathrm{Z}_{\mathfrak{F}(\mathcal{R})}(G)\leq M$ for every $\mathfrak{F}(\mathcal{R})$-maximal subgroup $M$ of $G$.
\end{proof}

\begin{proof}[\bf Proof of Theorem \ref{T3}] $(1)\Rightarrow(2)$.  Suppose that  $\mathrm{Z}_{\mathfrak{F}}(G)= \mathrm{Int}_{\mathfrak{F}}(G)$ holds for every group $G$ and
\begin{center} $\bigcup_{n=1}^m(\mathrm{Out}(G)\wr S_n\,|\,G\not\in\mathfrak{F}$ is a simple group and $n\in A_\mathcal{R}(G))\subseteq\mathfrak{F}$.
   \end{center}
     Let show that  $\mathrm{Int}_{\mathfrak{F}(\mathcal{R})}(G)=\mathrm{Z}_{\mathfrak{F}(\mathcal{R})}(G)$ also holds for every group $G$. Let    $\overline{H}=H/K$ be a chief factor of $G$ below $\mathrm{Int}_{\mathfrak{F}(\mathcal{R})}(G)$ and $\overline{G}=G/K$. Note that $G/C_G(\overline{H})\simeq \overline{G}/C_{\overline{G}}(\overline{H})$.

  $(a)$ \emph{If $\overline{H}\in\mathfrak{F}$, then it is $\mathfrak{F}(\mathcal{R})$-central in $G$.}

Directly follows from  Theorem \ref{t3}, $(3)$ of Definition \ref{MD1} and $\mathfrak{F}\subseteq \mathfrak{F}(\mathcal{R})$.

   $(b)$ \emph{If $\overline{H}\not\in\mathfrak{F}$, then $gr(\overline{H}, G)\in\mathcal{R}(\overline{H})$.}

Directly follows from   $(2)$ of Theorem \ref{t31}

 $(c)$ \emph{If $\overline{H}\not\in\mathfrak{F}$ and $n=r(\overline{H}, G)\in A_\mathcal{R}(\overline{H})$, then $\overline{H}$  is $\mathfrak{F}(\mathcal{R})$-central in $G$.}

Recall that $\overline{G}/C_{\overline{G}}(\overline{H})\overline{H}$ is isomorphic to a subgroup of $\mathrm{Out}(\overline{H})$ and $\overline{H}=\overline{H}_1\times\dots\times\overline{H}_n$ is the direct product of isomorphic simple groups. So $\overline{G}/C_{\overline{G}}(\overline{H})\overline{H}$ is isomorphic to a subgroup of $\mathrm{Out}(\overline{H}_i)\wr S_n\in\mathfrak{F}$. Since $\mathfrak{F}$ is hereditary, $\overline{G}/C_{\overline{G}}(\overline{H})\overline{H}\in\mathfrak{F}$.  Now $\overline{G}/C_{\overline{G}}(\overline{H})\in\mathfrak{F}(\mathcal{R})$ by Jordan-H\"{o}lder theorem and the definition of $\mathfrak{F}(\mathcal{R})$. So  $\overline{H}$  is $\mathfrak{F}(\mathcal{R})$-central in $G$ by Lemma \ref{l1}.

$(d)$ \emph{If $\overline{H}\not\in\mathfrak{F}$ and $n=r(\overline{H}, G)\in B_\mathcal{R}(\overline{H})$, then $\overline{H}$  is $\mathfrak{F}(\mathcal{R})$-central in $G$.}

 Now $\overline{H}=\overline{H}_1\times\dots\times\overline{H}_n$ is the direct product of isomorphic simple groups and every element of $G$ that fixes some $\overline{H}_i$ induces an inner automorphism on it. Hence $\overline{H}C_{\overline{G}}(\overline{H})=\cap_{i=1}^nN_{\overline{G}}(\overline{H}_i)$.  According to \cite[1.1.40(6)]{s9} $\overline{G}/C_{\overline{G}}(\overline{H})\overline{H}$ is isomorphic to a subgroup of $S_n$. From $\mathfrak{G}_{\{q\in\mathbb{P}\,|\,q\leq m\}}\subseteq\mathfrak{F}$ and $n\leq m$ it follows that $S_n\in\mathfrak{F}$. Since $\mathfrak{F}$ is hereditary, we see that $\overline{G}/C_{\overline{G}}(\overline{H})\overline{H}\in\mathfrak{F}$. Now $\overline{G}/C_{\overline{G}}(\overline{H})\in\mathfrak{F}(\mathcal{R})$ by Jordan-H\"{o}lder theorem and the definition of $\mathfrak{F}(\mathcal{R})$. Thus  $\overline{H}$  is $\mathfrak{F}(\mathcal{R})$-central in $G$ by Lemma \ref{l1}.

  $(c)$ \emph{$\mathrm{Z}_{\mathfrak{F}(\mathcal{R})}(G)= \mathrm{Int}_{\mathfrak{F}(\mathcal{R})}(G)$.}

  From   $(a)$ $(c)$ and $(d)$ it follows that   $\mathrm{Int}_{\mathfrak{F}(\mathcal{R})}(G)\leq\mathrm{Z}_{\mathfrak{F}(\mathcal{R})}(G)$.
  By Proposition \ref{p7}, $\mathrm{Z}_{\mathfrak{F}(\mathcal{R})}(G)\leq \mathrm{Int}_{\mathfrak{F}(\mathcal{R})}(G)$. Thus $\mathrm{Z}_{\mathfrak{F}(\mathcal{R})}(G)= \mathrm{Int}_{\mathfrak{F}(\mathcal{R})}(G)$.

  $(2)\Rightarrow(1)$. Assume that $\mathrm{Z}_{\mathfrak{F}(\mathcal{R})}(G)= \mathrm{Int}_{\mathfrak{F}(\mathcal{R})}(G)$ holds for every group $G$. Now $\mathfrak{F}(\mathcal{R})_l=\mathfrak{F}$ by Proposition \ref{t32}.  Hence $\mathrm{Z}_{\mathfrak{F}}(G)= \mathrm{Int}_{\mathfrak{F}}(G)$ holds for every group $G$ by $(2)$ of Theorems \ref{t3} and \ref{MT1}.

Assume that $G\not\in\mathfrak{F}$ is a simple group, $n\in A_\mathcal{R}(G)$. Let $T=\mathrm{Aut}(G)\wr S_n$ and $M$ be an $\mathfrak{F}(\mathcal{R})$-maximal subgroup of $T$. Note that     $T$ has a unique minimal normal subgroup $N$ which is isomorphic to the direct product of copies of $G$ and $r(N, T)=n$.

Let $L=NM$. Then $N$ is a direct product of minimal normal subgroups $N_i$ of $L$. Since $\mathcal{R}$ is very good generalized rank function, we see that $n\geq r(N_i, L)\in\mathcal{R}(N_i)$. From $L/N\simeq M\in\mathfrak{F}(\mathcal{R})$ it follows that $L\in\mathfrak{F}(\mathcal{R})$. So $M=L$. It means that $N\leq \mathrm{Int}_{\mathfrak{F}(\mathcal{R})}(T)=\mathrm{Z}_{\mathfrak{F}(\mathcal{R})}(T)$. Hence $T/C_T(N)\simeq T\in\mathfrak{F}(\mathcal{R})$. Thus $T/N\simeq \mathrm{Out}(G)\wr S_n\in\mathfrak{F}(\mathcal{R})$.
From $\mathfrak{G}_{\{q\in\mathbb{P}\,|\,q\leq m\}}\subseteq\mathfrak{F}$ it follows that $S_n\in\mathfrak{F}$. Note that $\mathrm{Out}(G)$ is soluble by Schreier conjecture. Hence $(T/N)^{E\mathfrak{F}}\simeq 1$. Now $T/N\in\mathfrak{F}$ by point $(a)$ of $(3)$ of Theorem 2.
\end{proof}

\begin{proof}[\bf Proof of Corollary \ref{c30}]
Let $\mathcal{R}(S)\equiv(\emptyset, 1)$. It is clear that $\mathcal{R}$ is very good. Recall that $\mathfrak{F}^*=\mathfrak{F}(\mathcal{R})$. Now Corollary \ref{c30} directly follows from Theorem \ref{T3}.
\end{proof}

\begin{proof}[\bf Proof of Corollary \ref{c31}]
  Assume that $\mathrm{Z}_{\mathfrak{N}(\mathcal{R})}(G)= \mathrm{Int}_{\mathfrak{N}(\mathcal{R})}(G)$ holds for every group $G$.

  Note that if every maximal subgroup of a noncyclic group $G$ is a $2$-group, the $G$ is nilpotent. From $S_3\in\mathcal{M}(\mathfrak{N})$ with $\mathrm{F}(S_3)=\tilde{\mathrm{F}}(S_3)$ it follows that $$ C_1=\min_{G\in\mathcal{M}(\mathfrak{N})\textrm{ with }\mathrm{F}(G)=\mathrm{\tilde{F}}(G)}\max_{M \textrm{ is a maximal subgroup of }G}|M|-1=2.$$

  Let $N$ be a simple non-abelian group.
  Then $\mathcal{R}(N)\subseteq[0, 2]$  by Theorem \ref{t31}. If $1\in A_\mathcal{R}(N)$ (resp. $2\in A_\mathcal{R}(N)$), then $\mathrm{Out}(N)$ (resp. $\mathrm{Out}(N)\wr S_2$) is nilpotent by Theorem \ref{T3}. Note that if $|\mathrm{Out}(N)|$ has an odd prime divisor $p$, then $\mathrm{Out}(N)\wr S_2$ contains a non-nilpotent dihedral subgroup with $2p$ elements. So if $\mathrm{Out}(N)\wr S_2$ is nilpotent, then it is a $2$-group.

  Assume now that for any simple non-abelian group $N$ holds:
$\mathcal{R}(N)\subseteq[0, 2]$ and if $1\in A_\mathcal{R}(N)$ (resp. $2\in A_\mathcal{R}(N)$), then $\mathrm{Out}(N)$ is nilpotent (resp. a  $2$-group).  Note that $\mathfrak{N}_2\subseteq\mathfrak{N}$. Then $\mathrm{Z}_{\mathfrak{N}(\mathcal{R})}(G)= \mathrm{Int}_{\mathfrak{N}(\mathcal{R})}(G)$ holds for every group $G$ by Theorem \ref{T3}.
\end{proof}

  \subsection*{Acknowledgments}

I am grateful to A.\,F. Vasil'ev for helpful discussions.

\end{document}